\documentclass[10pt,a4paper]{article}
\usepackage[utf8]{inputenc}
\usepackage{amsmath}
\usepackage{amsfonts}
\usepackage{amssymb}
\usepackage{amsthm}
\usepackage{graphicx}
\usepackage[left=2cm,right=2cm,top=2cm,bottom=2cm]{geometry}
\usepackage{mathrsfs}
\usepackage{enumerate}
\newtheorem{theorem}{Theorem}[section]
\newtheorem{lemma}[theorem]{Lemma}
\newtheorem{cor}[theorem]{Corollary}

\newtheorem{conj}[theorem]{Conjecture}

\begin{document}
\begin{center}
\Large{$(k,H)$-kernels in nearly tournaments}
\end{center}

\begin{center}
Hortensia Galeana-Sánchez and Miguel Tecpa-Galván
\end{center}

\begin{abstract}
Let $H$ be a digraph possibly with loops, $D$ a digraph  without loops, and $\rho : A(D) \rightarrow V(H)$ a coloring of $A(D)$ ($D$ is said to be an $H$-colored digraph). If  $W=(x_{0}, \ldots , x_{n})$ is a walk in $D$, and $i \in \{ 0, \ldots , n-1 \}$, we say that there is an obstruction on $x_{i}$ whenever $(\rho(x_{i-1}, x_{i}), \rho (x_{i}, x_{i+1})) \notin A(H)$ (when $x_{0} = x_{n}$ the indices are taken modulo $n$).
 We denote by $O_{H}(W)$ the set $\{ i \in \{0, \ldots , n-1 \} :$ there is an obstruction on $x_{i} \}$. The $H$-length of $W$, denoted by $l_{H}(W)$,  is defined by $|O_{H}(W)|+1$ whenever $x_{0} \neq x_{n}$, or $|O_{H}(W)|$ in other case. 

A $(k, H)$-kernel of an $H$-colored digraph $D$ ($k \geq 2$) is a subset of vertices of $D$, say $S$, such that, for every pair of different vertices in $S$, every path between them has $H$-length at least $k$, and for every vertex $x \in V(D) \setminus S$ there exists an $xS$-path with $H$-length at most $k-1$. This concept widely generalize previous nice concepts as kernel, $k$-kernel, kernel by monochromatic paths, kernel by properly colored paths, and $H$-kernel. 

In this paper, we will study the existence of $(k,H)$-kernels in interesting classes of digraphs, called nearly tournaments, which have been large and widely studied due its applications and theoretical results. We will show several conditions that guarantee the existence of $(k,H)$-kernel in tournaments, $r$-transitive digraphs, $r$-quasi-transitive digraphs, multipartite tournaments, and local tournaments.
\end{abstract}
\textbf{MSC-class}: 05C15, 05C20, 05C69.

\section{Introduction}

A digraph $D$ is \emph{semicomplete} if every two different vertices in $D$ are joined by at least one arc. 
A \emph{tournament} is a semicomplete digraph without symmetric arcs.
 The class of tournaments is one of the most well-studied classes of digraphs with many deeply and important results and applications, and provide a useful class of digraphs which allows of to have a first approach to solve very difficult problems.
As a consequence, several authors defined some new classes of digraphs which are generalizations of tournaments, commonly called nearly tournaments.
For instance, the class of \emph{local in-tournaments} (\emph{local out-tournaments}) was introduced by Bang-Jensen in \cite{localtour}, as a digraph $D$ such that for every $x \in V(D)$, the induced subdigraph of $N^{-}(x)$ (respectively $N^{+}(x)$) is a tournament. A \emph{local tournament} is a digraph which is both local in-tournament and local out-tournament. Several results related with local tournaments have been proved. For a deeply study of local tournaments, see \cite{bangjen}.  

Other digraphs related with nearly tournaments are the multipartite tournaments. An \emph{$r$-partite tournament} ($r \geq 2$) is a digraph $D$ such that $V(D)$ can be partitioned into $r$ disjoint independent sets, and every two vertices in different classes are joined by an asymmetric arc. Multipartite tournaments were considered by Moon in \cite{moon}, and several authors studied such class of digraphs.  For a deeply study of multipartite tournaments, see \cite{bangjen}. 

A digraph $D$ is \textit{$r$-transitive} ($r \geq 2$) if for every $\{u,v\} \subseteq V(D)$, whenever there exists a $uv$-path with length $r$, we have that $(u,v) \in A(D)$. When $r=2$, an $r$-transitive digraph is called \textit{transitive}. The class of $r$-transitive digraphs is one of the most studied nearly tournaments, due their particularly nice structure, and several results related with such digraphs arose. For instance, see \cite{bangjen}, \cite{ktrans3}, \cite{ktrans1} and \cite{ktrans2}.

A digraph $D$ is \textit{quasi-transitive} if for every $\{u,v\} \subseteq V(D)$, whenever there exists a $uv$-path with length $2$, we have that $u$ and $v$ are joined by an arc.
 Quasi-transitive digraphs were introduced by Ghouila-Houri in \cite{quasitrans} as a consequence of their relation with comparability graphs, and it is one of the most studied class of digraphs, probably the main reason is the characterization theorem showed by Bang-Jensen and J. Huang in \cite{qtranschar}. 
 Bang-Jensen \cite{3transintro} introduced the family of $3$-quasi-transitive digraphs in the context of strong arc-locally semicomplete digraphs. A digraph $D$ is \textit{$3$-quasi-transitive} if for every $\{u,v\} \subseteq V(D)$, whenever there exists a $uv$-path with length $3$, we have that $u$ and $v$ are joined by an arc. In an analogous way, a digraph $D$ is \emph{$r$-quasi-transitive} ($r \geq 2$) if for every $\{u,v\} \subseteq V(D)$, whenever there exists a $uv$-path with length $r$, we have that $u$ and $v$ are joined by an arc. This class  were introduced by Galeana-Sánchez and Hernández-Cruz \cite{qtrans1} in the context of $k$-kernels.  For a deeply study of $r$-quasi-transitive digraphs, see \cite{bangjen}. 

The concept of \emph{kernel} was introduced by von Neumann and Morgenstern in \cite{2} as a subset $S$ of vertices of a digraph $D$, such that for every pair of different vertices in $S$, there is no arc between them (that is, $S$ is an \emph{independent set}), and every vertex not in $S$ has at least one out-neighbor in $S$ (that is, $S$ is an \emph{absorbent set}). This concept has been deeply and widely studied by several authors due to a large amount of theoretical and practical applications, by example \cite{ker1}, \cite{ker2}, \cite{ker3} and \cite{ker4}. 
In \cite{3} Chvátal proved that deciding if a digraph has a kernel is an NP-complete problem. As a consequence, several conditions that guarantee the existence of kernels has been showed. For instance, we have the following two classical results:

\begin{theorem}[König \cite{kon}]
\label{kertrans}
Every transitive digraph has a kernel.
\end{theorem}

\begin{theorem}[Duchet \cite{duch}]
\label{duchet}
If $D$ is a digraph and every cycle in $D$ has a symmetric arc, then $D$ has a kernel.
\end{theorem}

The concept of kernel has been generalized over the years. A subset $S$ of vertices of $D$ is said to be a \emph{kernel by paths}, if for every $x \in V(D)\setminus S$, there exists an $xS$-path (that is, $S$ is \emph{absorbent by paths}) and, for every pair of different vertices $\{ u, v \} \subseteq S$, there is no $uv$-path in $D$ (that is, $S$ is \emph{independent by paths}). 
This concept was introduced by Berge in \cite{19}, and it is a well known result that every digraph has a kernel by paths \cite{19} (see Corollary 2 on p. 311).
 The concept of $(k,l)$-kernel was introduced by Borowiecki and Kwa\'snik in \cite{14} as follows:
 If $k \geq 2$, a subset $S$ of vertices of a digraph $D$ is a \emph{$k$-independent set}, if for every pair of different vertices in $S$, every walk between them has length at least $k$. If $l \geq 1$, we say that $S$ is an \emph{$l$-absorbent set} if for every $x \in V(D) \setminus S$ there exists an $xS$-walk with length at most $l$.
 If $k \geq 2$ and $l \geq 1 $, a \emph{$(k,l)$-kernel} is a subset of $V(D)$ which is both $k$-independent and $l$-absorbent.
  If $l=k-1$, the $(k,l)$-kernel is called \emph{$k$-kernel}. Notice that every $2$-kernel is a kernel. 
  A $(2,2)$-kernel is called \textit{quasi-kernel}, and it is known that every digraph has a quasi-kernel \cite{qker} (called semi-kernel by the authors).  Several sufficient conditions for the existence of $k$-kernels in nearly tournaments have been proved. For instance, see \cite{localtour1},  \cite{qtrans5}, \cite{qtrans1} and \cite{qtrans4}.

Let $H$ be a digraph possibly with loops, $D$ a digraph  without loops, and $\rho : A(D) \rightarrow V(H)$ a coloring of $A(D)$ ($D$ is said to be an $H$-colored digraph). 
A directed path in $D$ is an \emph{$H$-path}, whenever the consecutive colors encountered on $W$ form a directed walk in $H$
(the concepts of \emph{$H$-cycle} and \emph{$H$-walk} are defined analogously).
If $W=(x_{0}, \ldots , x_{n})$ is an $H$-path, $W$ is said to be an \emph{$x_{0}x_{n}$-$H$-path}. If $S \subseteq V(D)$ and $x_{n}\in S$, we say that $W$ is an \emph{$x_{0}S$-$H$-path}. 
In \cite{8} Linek and Sands introduced the concept of $H$-walk and their work was later considered by several authors, as example, \cite{7}, \cite{17}, \cite{12} and \cite{16}.
 A subset $S$ of vertices of $D$ is \emph{absorbent by $H$-paths}, if for every $x \in V(D) \setminus S$ there exists an $xS$-$H$-path; and $S$ is said to be \emph{independent by $H$-paths}, if for every pair of different vertices $\{ u, v \} \subseteq S$, there is no $uv$-$H$-path between them.
  A \emph{kernel by $H-$paths}, or simply $H$-kernel, is a subset of vertices of $D$ that is both absorbent by $H$-paths and independent by $H$-paths.
Several interesting kinds of kernels are particular cases of $H$-kernels, as example, kernels by paths, kernels by monochromatic paths, kernels by alternating paths, kernels by rainbow paths, and usual kernels. Several conditions on the existence of $H$-kernels have been showed, as example, see \cite{12} and \cite{18}.   
  
If $W=(x_{0}, \ldots , x_{n})$ is a walk in an $H$-colored digraph $D$, and $i \in \{0, \ldots , n-1 \}$, we say that there is an \emph{obstruction on $x_{i}$} iff $(\rho (x_{i-1}, x_{i}), \rho (x_{i}, x_{i+1})) \notin A(H)$ (indices are taken modulo $n$ if $x_{0}=x_{n}$). 
We denote by $O_{H}(W)$ the set $\{ i \in \{ 0, \ldots , n-1\} : \text{there is an obstruction on }x_{i} \}$. 
The \emph{$H$-length of $W$}, denoted by $l_{H}(W)$, is defined as  $l_{H}(W)= |O_{H}(W)|+1$ if $W$ is open, or  $l_{H}(W)= |O_{H}(W)|$ otherwise. 
The $H$-length was primarily studied by Galeana-Sánchez and Sánchez-L{\'o}pez in \cite{12} for closed walks, and by Andenmatten, Galeana-Sáchez and Pach in \cite{20} for open paths.  Clearly, the usual length $l(W)$ coincides with the $H$-length $l_{H}(W)$, in the very particular case when $A(H)=\emptyset$. An open walk in an $H$-colored digraph is an $H$-walk if and only if it has $H$-length 1.
A particular kind of $H$-coloring in graphs was studied by Szeider in \cite{21} and, as a consequence, in \cite{20} was proved that under the assumption P $\neq$ NP, finding $uv$-paths of minimum $H$-length in $H$-colored graphs ($H$-colored digraphs) has no polynomial solution (although there is a polynomial algorithm to find $uv$-paths of minimum $H$-length for some $H$, see \cite{20}).  

Let $D$ be an $H$-colored digraph and $S$ a subset of vertices of $D$. 
If $l \geq 1$, we say that $S$ is an \emph{$(l, H)$-absorbent set}, if for every $v \in V(D) \setminus S$ there exists a $vS$-path whose $H$-length is at most $l$.
 If $k \geq 2$, we say that $S$ is a \emph{$(k,H)$-independent set}, if for every pair of different vertices in $S$, every path between them has $H$-length at least $k$.
  If $k\geq 2$ and $l\geq l$, we say that $S$ is a \emph{$(k, l,H)$-kernel} if it is both $(k,H)$-independent and $(l,H)$-absorbent. 
If $l=k-1$, a $(k,l,H)$-kernel is called \emph{$(k, H)$-kernel}. Such concepts were introduced by Galeana-Sánchez and Tecpa-Galván in \cite{richardHkl}.
It is straightforward to see that every $H$-kernel is a $(2,H)$-kernel, and every $(k,l)$-kernel is a $(k,l,H)$-kernel if $H$ has no arcs nor loops. Since finding $(k,l)$-kernels in digraphs is a NP-complete problem, finding $(k,l,H)$-kernels in $H$-colored digraphs is also a NP-complete problem.

In this paper, we study the existence of $(k,H)$-kernels in $H$-colored nearly tournaments. For instance, we will prove that if $D$ is an $H$-colored digraph, $D$ has a $(k,H)$-kernel provided that:

\begin{enumerate}[(i)]
\item $k \geq 3$ and $D$ is a tournament.

\item $k \geq r \geq 2$ and $D$ is an $r$-transitive digraph.

\item $k \geq 4$ and $D$ is a quasi-transitive digraph.

\item $k \geq 5$ and $D$ is a $3$-quasi-transitive digraph.

\item $k \geq r \geq 2$, $D$ is an $r$-quasi-transitive digraph, and every cycle with length $r+1$ is an $H$-cycle.

\item $k \geq 5$ and $D$ is an $r$-partite tournament ($r \geq 2$)

\item $D$ is a local in-tournament (out-tournament) and every cycle has $H$-length at most $k-2$.
\end{enumerate}

\section{First results}

For terminology and notation not defined here, we refer the reader to  \cite{1}. Two vertices in a digraph are \textit{adjacent} if there exist an arc between them.  An arc $(u,v)$ in a digraph $D$ is \emph{symmetric} if $(v,u) \in A(D)$, otherwise, such arc is called \emph{asymmetric}. In this paper we write walk, path and cycle, instead of directed walk, directed path, and directed cycle, respectively.
 If $W=(x_{0}, \ldots , x_{n})$ is a walk (path), we say that $W$ is an $x_0x_n$-\emph{walk} ($x_0x_n$-\emph{path}). If $n \geq 2$ and $i \in \{ 1, \ldots , n-1\}$, we say that $x_{i}$ is an \emph{internal vertex of $W$}.  The \emph{length} of $W$ is the number $n$ and it is denoted by $l(W)$. A cycle with length $k$ is called \emph{$k$-cycle.}
 
If $T_{1}=(z_{0}, \ldots , z_{n})$ and $T_{2}=(w_{0}, \ldots , w_{m})$ are walks in a digraph $D$, and $z_{n}=w_{0}$, we denote by $T_{1} \cup T_{2}$ the walk $(z_{0}, \ldots , z_{n} = w_{0}, \ldots , w_{m})$.
 Let $\{ v_{0}, \ldots , v_{n} \}$ be a subset of vertices of $D$, and for every $i \in \{0, \ldots , n-1\}$, $T_{i}$ a $v_{i}v_{i+1}$-walk, we denote by $\cup_{i=0}^{n-1} T_{i}$ the concatenation of such walks.
  Given $W=(x_{0}, \ldots x_{n})$  a walk in $D$, and $\{i, j \} \subseteq \{0, \ldots , n-1 \}$ with $i < j$, we denote by $(x_{i},W, x_{j})$ the walk $(x_{i}, x_{i+1}, \ldots , x_{j})$.
If $S_1$ and $S_2$ are two disjoint subsets of  $V(D)$, a $uv$-walk  in  $D$ is called an $S_1S_2$-\emph{walk} whenever $u$ $\in$ $S_1$ and $v$ $\in$ $S_2$. If $S_1 = \{x\}$ or $S_2 = \{x\}$, then we write $xS_2$-\emph{walk} or $S_1x$-\emph{walk}, respectively. If there exists at least one $uv$-path in $D$, a $uv$-path with minimum length is called \emph{$uv$-geodesic}, and its length is denoted by $d_{D}(u,v)$. 

Let $D$ be an $H$-colored digraph and $k \geq 2$. We define the \textit{$(k-1,H)$-closure of $D$}, denoted by $C_{H}^{k-1}(D)$, as the digraph such that $V(C_{H}^{k-1}(D)) = V(D)$ and $(u,v) \in A(C_{H}^{k-1} (D))$ iff there exists a $uv-$path in $D$ with $H$-length at most $k-1$. The following lemmas will be useful in what follows.

\begin{lemma}
\label{walks}
Let $D$ be an $H$-colored digraph and $k \geq 2$. The following assertions holds:
	\begin{enumerate}[a)]
	\item If $W$ is an open walk in $D$, then $l_{H}(W) \leq l(W)$.
	
	\item If $\{u, v \} \subseteq V(D)$, and there exists a $uv-$walk with length at most $k-1$ in $D$, then $(u,v) \in A(C_{H}^{k-1} (D) )$. 
	
	\item If $\{u, v \} \subseteq V(D)$, and there exists a closed walk in $D$ with length at most $k$, say $W$, such that $\{u, v \} \subseteq V(W)$, then $(u,v)$ is a symmetric arc in $C_{H}^{k-1}(D)$.
	\end{enumerate}
\end{lemma}
\begin{proof}
It follows from the definition of $H$-length that (a) holds. On the other hand, from Lemma \ref{walks} (a), and  the definition of $(k-1,H)$-closure, we have that (b) holds. In order to show that (c) holds, notice that two different vertices in $W$ are joined by a subpath of $W$ with length at most $k-1$. By applying Lemma \ref{walks} (b), we can conclude (c).
\end{proof}

\begin{lemma}
\label{subpaths}
Let $D$ be an $H$-colored digraph, and $W=(x_{0}, \ldots , x_{n})$ a cycle in $D$. If $\{ i, j \} \subseteq \{ 0, \ldots , n \}$ with $i < j$, and $W'=(x_{i}, W , x_{j})$ then $O_{H}(W') \subseteq O_{H}(W)$. 
\end{lemma}
\begin{proof}
Clearly, Lemma \ref{subpaths} holds when $O_{H}(W') = \emptyset$. Now, suppose that $O_{H}(W') \neq \emptyset$ and let $t \in O_{H}(W')$. It follows from the definition of $O_{H}(W')$ that $t \in \{i+1 \ldots , j-1 \}$ and $(\rho (x_{t-1},  x_{t}) , \rho (x_{t}, x_{t+1})) \notin A(H)$. Hence, $t \in \{0, \ldots , n-1 \}$ and $(\rho (x_{t-1},  x_{t}) , \rho (x_{t}, x_{t+1})) \notin A(H)$, which implies that $t \in O_{H}(W)$. Therefore, $O_{H}(W') \subseteq O_{H}(W)$. 
\end{proof}

 Given $k \geq 2$, a natural relation between $(k,H)$-kernels in $H$-colored digraphs and kernels in the $(k-1,H)$-closure is showed in the following lemma.

\begin{lemma}
\label{closure}
Let $D$ be an $H$-colored digraph and  $k \geq 2$. $D$ has a $(k,H)$-kernel if and only if $C_{H}^{k-1}(D)$ has a kernel.
\end{lemma}
\begin{proof}
First we will show that if $D$ has a $(k,H)$-kernel, say $S$, then $S$ is a kernel in $C^{k-1}_{H}(D)$.

In order to show that $S$ is an absorbent set in $C^{k-1}_{H}(D)$, consider $x \in V(C^{k-1}_{H}(D)) \setminus S$. Since $S$ is a $(k,H)$-kernel in $D$,  there exists $w \in S$, and an $xw$-path in $D$ with $H$-length at most $k-1$. It follows from the definition of $C^{k-1}_{H}(D)$ that $(x,w) \in A(C^{k-1}_{H}(D))$, which implies that $S$ is an absorbent set in $C^{k-1}_{H}(D)$.

Now, we will show that $S$ is an independent set in $C^{k-1}_{H}(D)$. Proceeding by contradiction, suppose that there exists $\{ u, v \} \subseteq S$ such that $(u,v) \in A(C^{k-1}_{H}(D))$. It follows from definition of $C^{k-1}_{H}(D)$ that there exists a $uv$-path with length at most $k-1$ in $D$, which is no possible since $S$ is a $(k,H)$-independent set in $D$. Hence, $S$ is an independent set in $C^{k-1}_{H}(D)$. Therefore, if  $D$ has a $(k,H)$-kernel, then $C_{H}^{k-1}(D)$ has a kernel.

Now, we will show that if $C^{k-1}_{H}(D)$ has a kernel, say $S'$, then $S'$ is a $(k,H)$-kernel in $D$.

In order to show that $S'$ is a $(k-1,H)$-absorbent set in $D$, consider $x \in V(D) \setminus S'$. Since $S'$ is a kernel in $C^{k-1}_{H}(D)$, there exists $w \in S'$ such that $(x,w) \in A(C^{k-1}_{H}(D))$. Hence, by the definition of $C^{k-1}_{H}(D)$, there exists an $xw$-path in $D$ with $H$-length at most $k-1$, which implies that $S'$ is a $(k-1,H)$-absorbent set in $D$.

Now, we will show that $S'$ is a $(k,H)$-independent set in $D$. Proceeding by contradiction, suppose that there exists $\{ u, v \} \subseteq S'$ and a $uv$-path in $D$, with $H$-length at most $k-1$. It follows from definition of $C^{k-1}_{H}(D)$ that $(u,v) \in A(C^{k-1}_{H}(D))$, which is no possible since $S'$ is an independent set in $C^{k-1}_{H}(D)$. Hence, $S'$ is a $(k,H)$-independent set in $C^{k-1}_{H}(D)$. Therefore, if  $C^{k-1}_{H}(D)$ has a kernel, then $D$ has a $(k,H)$-kernel.
\end{proof}

\section{$(k,H)$-kernels in nearly tournaments.}

In this section, we will show several conditions in $H$-colored nearly tournaments that guarantee the existence of $(k,H)$-kernels for certain values of $k$.

\subsection{Tournaments and semicomplete digraphs}

\begin{theorem}
\label{tour1}
Let $D$ be an $H$-colored semicomplete digraph. If $k \geq 3$, then $D$ has a $(k,H)$-kernel.
\end{theorem}
\begin{proof}
Let $S$ be a quasi-kernel of $D$. Since $D$ is semicomplete, then $S=\{ x \}$ for some $x \in V(D)$. Clearly, $S$ is a $(k,H)$-independent set in $D$. On the other hand, since every vertex $w \in V(D) \setminus S$ holds that $d_{D}(w,x) \leq 2$, it follows from Lemma \ref{walks} (a) that  $S$ is a $(k-1,H)$-absorbent set in $D$. Therefore, $S$ is a $(k,H)$-kernel in $D$.
\end{proof}

\begin{theorem}
\label{tour2}
Let $D$ be an $H$-colored semicomplete digraph in which every $3$-cycle is an $H$-cycle. If $k \geq 2$, then $D$ has a  $(k,H)$-kernel.
\end{theorem}
\begin{proof}
It follows from Theorem \ref{tour1} that only remains to show that $D$ has a kernel by $H$-paths. Consider a quasi-kernel of $D$, say $S$.  Since $D$ is semicomplete, then $S=\{ x \}$ for some $x \in V(D)$. 
Clearly, $S$ is an independent set by $H$-paths. 
In order to show that $S$ is an absorbent set by $H$-paths, consider $w \in V(D) \setminus S$.
By the choice of $S$, we have that either $d_{D}(w,x)=1$ or $d_{D}(w,x)=2$.  If $d_{D}(w,x)=1$, then $(w,x)$ is a $wx$-$H$-path. If $d_{D}(w,x)=2$, consider a $wx$-geodesic in $D$, say $(w,z,x)$. Since $D$ is semicomplete, then $w$ and $x$ are joined by an arc in $D$ and, since $d_{D}(w,x)=2$, it follows that $(x,w) \in A(D)$. Hence, $C=(w,z,x,w)$ is a $3$-cycle in $D$ and, by hypothesis, $C$ is an $H$-cycle. We can conclude that $(w,z,x)$ is a $wx$-$H$-path. Therefore, $S$ is an absorbent set by $H$-paths in $D$.
\end{proof}

From Theorem \ref{tour1} and Theorem \ref{tour2}, it is straightforward to see the following corollaries. 

\begin{cor}
If $D$ is an $H$-colored tournament, then for every $k \geq 3$, $D$ has a $(k,H)$-kernel.
\end{cor}

\begin{cor}
If $D$ is an $H$-colored tournament, and every $3$-cycle in $D$ is an $H$-cycle, then for every $k \geq 2$, $D$ has a $(k,H)$-kernel.
\end{cor}

\subsection{$r$-transitive digraphs}

The following lemma will be useful in what follows.

\begin{lemma}\cite{3qt}
\label{ktrans}
Let $D$ be an $r$-transitive digraph ($r \geq 2$), and $\{u, v \} \subseteq V(D)$. If there exists a $uv$-path in $D$, then $d_{D}(u,v) \leq r-1$.
\end{lemma}

As a consequence of the previous lemma, we have the following results.

\begin{theorem}
Let $D$ be an $H$-colored $r$-transitive digraph ($r \geq 2$). For every $l \geq r-1$ and $k \geq 2$, $D$ has a $(k, l, H)$-kernel.
\end{theorem}
\begin{proof}
Let $N$ be a kernel by paths in $D$. Since $N$ is a path-independent set, then for every $k \geq 2$, $N$ is a $(k,H)$-independent set. On the other hand, consider $x \in V(D) \setminus N$. Since $N$ is a kernel by paths in $D$, there exists $w \in N$ and an $xw$-path. If $W$ is an $xw$-geodesic, then $l(W) \leq r-1$ (Lemma \ref{ktrans}), which implies that $l_{H}(W) \leq r-1$ (Lemma \ref{walks} (b)). Hence, $l_{H}(W) \leq l$, concluding that $N$ is an $(l,H)$-absorbent set. Therefore, $N$ is a $(k,l,H)$-kernel. 
\end{proof}

\begin{cor}
\label{panrtrans}
Let $D$ be an $H$-colored $r$-transitive digraph. For every $k \geq r$, $D$ has a $(k,H)$-kernel.
\end{cor}

\begin{cor}
\label{pantrans}
Let $D$ be an $H$-colored transitive digraph. If $k \geq 2$, then $D$ has a $(k,H)$-kernel.
\end{cor}

\subsection{Quasi-transitive digraphs.}

The following lemma provides a nice structure on quasi-transitive digraphs, and will be useful to guarantee the existence of $(k,H)$-kernels for certain values of $k$ in $H$-colored quasi-transitive digraphs.

\begin{lemma}\cite{1}
\label{quasibang}
Let $D$ be a quasi-transitive digraph and $\{u, v \} \subseteq V(D)$ such that there exists a $uv$-path in $D$. If $u$ and $v$ are not adjacent in $D$, then there exists $\{ x,z \} \subseteq V(D) \setminus \{u,v \}$ such that $(u,x)$, $(x,z)$, $(z,v)$, $(z,u)$ and $(v,x)$ are arcs of $D$.
\end{lemma}

As a consequence of the previous lemma, we have the following result.

\begin{lemma}
\label{quasi}
Let $D$ be an $H$-colored quasi-transitive digraph, and $k \geq 4$. If there exists a $uv$-walk in $D$, then either $u$ and $v$ are adjacent in $D$ or $(v,u)$ is a symmetric arc in $C_{H}^{k-1}(D)$.
\end{lemma}
\begin{proof}
If $(u,v)$ or $(v,u)$ is an arc of $D$, we are done. Now, suppose that $u$ and $v$ are not adjacent in $D$. It follows from Lemma \ref{quasibang} that there exists $\{ x,z \} \subseteq V(D) \setminus \{u,v \}$ such that $(u,x)$, $(x,z)$, $(z,v)$, $(z,u)$ and $(v,x)$ are all arcs of $D$. Hence, $(u,x,z,v)$ and $(v,x,z,u)$ are paths in $D$ with length 3. By Lemma \ref{walks} (b), we can conclude that $(u,v)$ is a symmetric arc in $C_{H}^{k-1}(D)$.  
\end{proof}

\begin{theorem}
\label{panqker}
Let $D$ be an $H$-colored quasi-transitive digraph. If $k \geq 4$, then $D$ has a $(k,H)$-kernel.
\end{theorem}
\begin{proof}
First, we will show that every cycle in $C_{H}^{k-1}(D)$ has a symmetric arc, then, by applying Theorem \ref{duchet} and Lemma \ref{closure}, we will conclude that $D$ has a $(k,H)$-kernel. Proceeding by contradiction, suppose that there exist a cycle in  $C_{H}^{k-1}(D)$, say $C=(u_{0}, \ldots , u_{n})$, with no symmetric arcs. Clearly, $n \geq 3$.

It follows from Lemma \ref{quasi} that, for every $i \in  \{ 0, \ldots , n-1\}$, $u_{i}$ and $u_{i+1}$ are adjacent in $D$ (indices are taken modulo $n$). Moreover, since $C$ has no symmetric arcs in $C_{H}^{k-1}(D)$, we can conclude that  for every $i \in  \{ 0, \ldots , n-1\}$, $(u_{i}, u_{i+1}) \in A(D)$, which implies that $C$ is a cycle in $D$. Let $q = max \{ i \in \{1, \ldots , n-1 \} : (u_{0}, u_{i}) \in A(D) \}$. 

Notice that $q \leq n-2$, otherwise, $(u_{n-1},u_{0})$ is a symmetric arc in $D$, which implies that $C$ has a symmetric arc in $C_{H}^{k-1}(D)$, contradicting the choice of $C$. Hence, $(u_{0}, u_{q}, u_{q+1})$ is a path in $D$.
 Since $D$ is a quasi-transitive digraph, we have that either $(u_{q+1}, u_{0}) \in A(D)$ or $(u_{0}, u_{q+1}) \in A(D)$. By the choice of $q$, it follows that $(u_{q+1}, u_{0})  \in A(D)$.
  We can conclude that $(u_{q+1}, u_{0}, u_{q})$ is a path in $D$ with length at most $k-1$, and, by Lemma \ref{walks} (b), we have that $(u_{q+1}, u_{q}) \in A(C_{H}^{k-1}(D))$, contradicting the choice of $C$. Therefore, every cycle in $C_{H}^{k-1}(D)$ has a symmetric arc.

It follows from Theorem \ref{duchet} that $C_{H}^{k-1}(D)$ has a kernel, which implies that $D$ has a $(k,H)$-kernel (Lemma \ref{closure}).
\end{proof}

\subsection{3-quasi-transitive digraphs}

The following lemma for $3$-quasi-transitive digraphs will be useful.

\begin{lemma}\cite{3qt}
\label{distance3quasitrans}
Let $D$ be a $3$-quasi-transitive digraph, and $\{u,v \} \subseteq V(D)$ such that there exists a $uv$-walk in $D$. The following assertions holds:
	\begin{enumerate}[a)]
	\item If $d(u,v) = 3$ or $d(u,v) \geq 5$, then $d(v,u)=1$.
	
	\item If $d(u,v)=4$, then $d(v,u) \leq 4$.
	\end{enumerate}
\end{lemma}

By the previous lemma, we have the following corollary.

\begin{cor}
\label{geo3qt}
Let $D$ be an $H$-colored $3$-quasi-transitive digraph, $k \geq 5$, and $(u,v) \in A(C_{H}^{k-1}(D))$. If $(u,v)$ is an asymmetric arc in $C_{H}^{k-1}(D)$, then $d_{D}(u,v) \leq 2$.
\end{cor}
\begin{proof}
Proceeding by contradiction, suppose that $d_{D}(u,v) \geq 3$. It follows from Lemma \ref{distance3quasitrans} that there exists a $vu$-path with length at most $4$ in $D$. By Lemma \ref{walks} (b), we can conclude that $(v,u) \in A(C_{H}^{k}(D))$, which is no possible since $(u,v)$ is an asymmetric arc in $C_{H}^{k-1}(D)$. Hence, $d_{D}(u,v) \leq 2$. 
\end{proof}

A more elaborated proof will show that every $H$-colored $3$-quasi-transitive digraph has a $(k,H)$-kernel for every $k \geq 5$.

\begin{theorem}
\label{pan3qtr}
Let $D$ be an $H$-colored $3$-quasi-transitive digraph. For every $k \geq 5$, $D$ has a $(k,H)$-kernel.
\end{theorem}
\begin{proof}
First, we will show that every cycle in $C_{H}^{k-1}(D)$ has a symmetric arc, then, by applying Theorem \ref{duchet} and Lemma \ref{closure}, we will conclude that $D$ has a $(k,H)$-kernel. Proceeding by contradiction, suppose that there exist a cycle in  $C_{H}^{k-1}(D)$, say $C=(u_{0}, \ldots , u_{n})$, with no symmetric arcs. Clearly, $n \geq 3$. 

For every $i \in \{0, \ldots , n-1\}$, consider a $u_{i}u_{i+1}$-geodesic in $D$, say $W_{i}$ (indices are taken modulo $n$). It follows from Corollary \ref{geo3qt} that for every $i \in \{ 0, \ldots , n-1\}$, $l(W_{i}) \leq 2$. Let $C'=\cup_{i=0}^{n-1}W_{i}$, and assume that $C' =(x_{0}, \ldots , x_{l})$. Notice that $C'$ is a closed walk in $D$, and $x_{0}=u_{0}=u_{l}$.

\begin{description}
\item \textbf{Claim 1.} For every $i \in \{ l-3,l-2,l-1, l\}$, $(u_{1}, x_{i}) \notin A(D)$.

Proceeding by contradiction, suppose that there exists  $i \in \{ l-3, l-2,l-1,l \}$, such that $(u_{1}, x_{i}) \in A(D)$. Since $(u_{0}, u_{1})$ is an asymmetric arc in $C^{k-1}_{H}(D)$, then $i \neq l$.  It follows that $(u_{1},x_{i}) \cup (x_{i}, C', x_{0})$ is a $u_{1}u_{0}$-walk in $D$ with length at most 4, which implies that $(u_{1}, u_{0}) \in A(C_{H}^{k-1}(D))$ (Lemma \ref{walks} (b)), contradicting the choice of $C$, and the claim holds.
\end{description}

On the other hand, since $u_{1} \in V(C')$, then $u_{1}$ has at least one out-neighbor in $V(C')$, and let $q=max \{ i \in \{ 1, \ldots , l \} : (u_{1}, x_{i}) \in A(D)\}$. The following assertions will be useful:
\begin{enumerate}[(i)]
\item $q \leq l -4$.

It follows from Claim 1.

\item For every $t \in \{q+1, q+2, q+3, q+4\}$, $x_{q} \neq x_{t}$.

It follows from the choice of $q$.

\item For every $t \in \{ q, q+1, q+2, q+3, q+4\}$, $u_{1} \neq x_{t}$.

If $q < l-4$, then by the choice of $q$, we have that $u_{1} \neq x_{i}$ for every $i \in \{ q, q+1, q+2, q+3, q+4\}$. If $q=l-4$ (that is, $x_{q+4}=x_{l}$), then by the choice of $q$, we have that $u_{1} \neq x_{i}$ for every $i \in \{ q, q+1, q+2, q+3\}$ and, since $u_{0} = x_{l}$, we have that $u_{1} \neq x_{l}$.
\end{enumerate}

It follows from (ii) and (iii) that $(u_{1}, x_{q}, x_{q+1}, x_{q+2})$ is a path in $D$. Since $D$ is a $3$-quasi-transitive digraph, then either $(u_{1},x_{q+2}) \in A(D)$ or $(x_{q+2}, u_{1}) \in A(D)$. By the choice of $q$, we have that $(x_{q+2}, u_{1}) \in A(D)$.  

\begin{description}
\item \textbf{Claim 2.} There exists $t \in \{0, \ldots , n-1\}$ such that $x_{q+1} = u_{t}$ and $x_{q+3} = u_{t+1}$.

First, we will show that  $(x_{q}, x_{q+1}, x_{q+2})$ has no two consecutive vertices of $C$...(*). Proceeding by contradiction, suppose that $(x_{q}, x_{q+1}, x_{q+2})$ has two consecutive vertices of $C$, say $u_{r}$ and $u_{r+1}$. Since $(u_{1}, x_{q}, x_{q+1}, x_{q+2}, u_{1})$ is a $4$-cycle in $D$, it follows from Lemma \ref{walks} (c) that $u_{r}$ and $u_{r+1}$ are joined by a symmetric arc in $C_{H}^{k-1}(D)$, which is no possible by the choice of $C$. Therefore, (*) holds.

On the other hand, from the choice of $C'$, we can conclude that there exists $t \in \{0, \ldots , n-1\}$ such that $A(W_{t}) \cap \{ (x_{q}, x_{q+1}),(x_{q+1}, x_{q+2}) \} \neq \emptyset$. Remark that $l(W_{t}) \leq 2$. It follows from (*) that $x_{q+1} = u_{t}$ and $x_{q+2} = u_{t+1}$, and the claim holds.
\end{description}

Now, we will show that $(x_{q}, x_{q+1}, \ldots , x_{q+4})$ is a path in $D$. Remark that $x_{q} \notin \{ x_{q+1}, x_{q+2}, x_{q+3}, x_{q+4}\}$ (because of (ii)). In order to show that  $x_{q+1} \notin \{ x_{q+2}, x_{q+3}, x_{q+4}\}$, notice that, since $(x_{q+1}, x_{q+2}) \in A(D)$, then $x_{q+1} \neq x_{q+2}$. Since $u_{t} = x_{q+1}$ and $u_{t+1}=x_{q+3}$ (Claim 2), then $x_{q+1} \neq x_{q+3}$. Now, $x_{q+1} \neq x_{q+4}$, otherwise, $(u_{t+1}, u_{t}) \in A(C_{H}^{k-1}(D))$, which is no possible by the choice of $C$. In order to show that  $x_{q+2} \notin \{ x_{q+3}, x_{q+4}\}$, we have that $x_{q+2} \neq x_{q+3}$ because they are joined by an arc in $D$. Moreover, we can see that $x_{q+2} \neq  x_{q+4}$, otherwise, $(u_{t+1}, x_{q+2}, u_{1}, x_{q}, u_{t})$ is a $u_{t+1}u_{t}$-walk in $D$ with length at most $4$, and by Lemma \ref{walks} (b), $(u_{t+1}, u_{t}) \in A(C_{H}^{k-1}(D))$, which is no possible by the choice of $C$. Finally, $x_{q+3} \neq x_{q+4}$ because they are joined by an arc in $D$. Hence, $(x_{q}, x_{q+1}, \ldots , x_{q+4})$ is a path in $D$.

Since $D$ is a $3$-quasi-transitive digraph, then we have that either $(x_{q}, x_{q+3}) \in A(D)$ or $(x_{q+3}, x_{q}) \in A(D)$. If $(x_{q+3}, x_{q})  \in A(D)$, then $(x_{q+3}, x_{q}, x_{q+1})$ is a $u_{t+1}u_{t}$-path in $D$ with length at most $k-1$, which implies that $(u_{t+1}, u_{t}) \in A(C_{H}^{k-1}(D))$, contradicting the choice of $C$. Hence, $(x_{q}, x_{q+3}) \in A(D)$. It follows that $(u_{1}, x_{q}, x_{q+3}, x_{q+4})$ is a $3$-path in $D$, concluding that either $(u_{1}, x_{q+4}) \in A(D)$ or $(x_{q+4}, u_{1}) \in A(D)$. By the choice of $q$, we have that $(x_{q+4}, u_{1}) \in A(D)$. In that case, $(x_{q+3}, x_{q+4}, u_{1}, x_{q}, x_{q+1})$ is a $u_{t+1}u_{t}$-path in $D$ with length at most $k-1$, concluding that $(u_{t+1}, u_{t}) \in A(C_{H}^{k-1}(D))$ (Lemma \ref{walks} (b)), contradicting the choice of $C$.

Therefore, every cycle in $C_{H}^{k-1}(D)$ has a symmetric arc, concluding that $C_{H}^{k-1}(D)$ has a kernel (Theorem \ref{duchet}), which implies that $D$ has a $(k,H)$-kernel (Lemma \ref{closure}).
\end{proof}

It follows from Theorem \ref{pan3qtr} and Theorem \ref{pan3qtr} the following conjecture.

\begin{conj}
 If $D$ is an $H$-colored $r$-quasi-transitive digraph ($r \geq 2$), then $D$ has a $(k,H)$-kernel for every $k \geq r+2$.
\end{conj}

\subsection{$r$-quasi-transitive digraphs}

The Theorem \ref{tqtrc} will show that, under certain conditions on the cycles of an $H$-colored $r$-quasi-transitive digraph, it is possible to guarantee the existence of $(k,H)$-kernels in such digraphs. First, we have the following result.

\begin{theorem}
\label{cycleskquasitrans}
Let $D$ be an $H$-colored $r$-quasi-transitive digraph ($r \geq 2$), such that every cycle in $D$ with length $r+1$ is an $H$-cycle, and $\{ u, v \} \subseteq V(D)$. If $d_{D}(u,v) \geq r$ and $T$ is a $uv$-geodesic in $D$, then $T$ is an $H$-path.
\end{theorem}
\begin{proof}
Let $T=(u=x_{0}, \ldots , x_{n} = v)$. If $n=r$, it follows from the fact that $D$ is an $r$-quasi-transitive digraph that either $(u,v) \in A(D)$ or $(v,u) \in A(D)$. Since $d_{D}(u,v)\geq r$, then we have that $( v,u) \in A(D)$. Hence, $C' = T \cup (v,u)$ is a cycle with length $r+1$. By hypothesis, $C'$ is an $H$-cycle, which implies that $T$ is an $H$-path in $D$.

Now, suppose that $n > r$. In order to show that $T$ is an $H$-path, we will show that $T$ has no obstructions. Let $i \in \{1, \ldots , n-1\}$, and consider the following cases:
	\begin{description}
	\item \textbf{Case 1.} $i \in \{ 1, \ldots, n-r \}$.
	
	In this case, notice that $(x_{i-1}, x_{i}, \ldots , x_{i+r-1})$ is a path with length $r$. Since $D$ is an $r$-quasi-transitive digraph, then either $(x_{i-1}, x_{i+r-1}) \in A(D)$ or $(x_{i+r-1}, x_{i-1})\in A(D)$. It follows from the fact that $T$ is a $uv$-geodesic, that $(x_{i-1}, x_{i+r-1}) \notin A(D)$, which implies that $(x_{i+r-1}, x_{i-1})\in A(D)$. Hence, $(x_{i-1},x_{i},  \ldots , x_{i+r-1}, x_{i-1})$ is a cycle with length $r+1$, which is an $H$-cycle by hypothesis. Therefore, $(\rho (x_{i-1}, x_{i}), \rho (x_{i}, x_{i+1})) \in A(H)$.
	
	\item \textbf{Case 2.} $i \in \{ n-r+1, \ldots , n \}$
	
	In this case, notice that $(x_{n-r} , \ldots , x_{i-1}, x_{i}, x_{i+1}, \ldots ,  x_{n})$ is a path with length $r$. Since $D$ is an $r$-quasi-transitive digraph, then either $(x_{n-r}, x_{n}) \in A(D)$ or $(x_{n}, x_{n-r})\in A(D)$. It follows from the fact that $T$ is a $uv$-geodesic, that $(x_{n-r}, x_{n}) \notin A(D)$, which implies that $(x_{n}, x_{n-r})\in A(D)$. Hence, $(x_{n-r}, x_{n-r+1}, \ldots , x_{n}, x_{n-r})$ is a cycle with length $ r+1$, which is an $H$-cycle by hypothesis. Therefore, $(\rho (x_{i-1}, x_{i}), \rho (x_{i}, x_{i+1})) \in A(H)$
	\end{description}
By the previous cases, we conclude that $T$ is an $H$-path.
\end{proof}

\begin{theorem}
\label{tqtrc}
Let $D$ be an $H$-colored $r$-quasi-transitive digraph ($r \geq 2$), such that every cycle with length $r+1$ is an $H$-cycle. For every $k \geq r$, $D$ has a $(k,H)-$kernel.
\end{theorem}
\begin{proof}
First, we will show that for every $k \geq r$, $C_{H}^{k-1}(D)$ is a transitive digraph. Then, by applying Theorem \ref{kertrans}, we can conclude that $C_{H}^{k-1}(D)$ has a kernel and, by Lemma \ref{closure}, $D$ has a $(k,H)$-kernel.

Consider $(u,v)$ and $(v,w)$ in $A(C_{H}^{k-1}(D))$ with $u \neq w$. It follows that there exists a $uv$-path in $D$ with $H$-length at most $k-1$, say $T_{1}$, and there exists a $vw$-path in $D$ with $H$-length at most $k-1$, say $T_{2}$. Hence, $T_{1} \cup T_{2}$ is a $uw-$walk in $D$. Now, consider a $uw$-geodesic in $D$, say $T$. If $l(T) \leq k-1$, then $(u,w) \in A(C_{H}^{k-1}(D))$ (Lemma \ref{walks} (b)). If $l(T) \geq k$, then by Theorem \ref{cycleskquasitrans} we have that $T$ is a $uw$-$H$-path, which implies that $(u,w) \in A(C_{H}^{k-1}(D))$. 

Therefore, $C_{H}^{k-1}(D)$ is a transitive digraph, which implies that $C_{H}^{k-1}(D)$ has a kernel (Theorem \ref{kertrans}). We can conclude that $D$ has a $(k, H)$-kernel (Lemma \ref{closure}). 
\end{proof}

As a consequence of the previous theorem, we have the following corollary for quasi-transitive digraphs.

\begin{cor}
Let $D$ be a quasi-transitive digraph such that every $3$-cycle is an $H$-cycle. For every $k \geq 2$, $D$ has a $(k, H)$-kernel. 
\end{cor}

\subsection{Multipartite tournaments}

The following lemma will be useful in order to show that every $H$-colored multipartite tournament has a $(k,H)$-kernel, for certain values of $k$.

\begin{lemma}
\label{distrpartite}
Let $D$ be an $H$-colored $r$-partite tournament, $k \geq 5$, and $(u,v) \in A(C_{H}^{k-1}(D))$. If $(u,v)$ is an asymmetric arc in $C_{H}^{k-1}(D)$, then $d_{D}(u,v) \leq 2$.
\end{lemma}
\begin{proof}
Let $\{ S_{1}, \ldots , S_{r} \}$ be the partition of $V(D)$ into independent sets. Proceeding by contradiction, suppose that $d_{D}(u,v) \geq 3$, and let $W=(u=x_{0}, \ldots , x_{n}=v)$ be a $uv-$geodesic in $D$. If $u \in S_{i}$ and $v \in S_{j}$ with $i \neq j$, then either $(u,v) \in A(D)$ or $(v,u) \in A(D)$. Since $d_{D}(u,v) \geq 3$, we conclude that $(v,u) \in A(D)$, which implies that $(v,u) \in A(C_{H}^{k-1}(D))$, contradicting the assumption on $(u,v)$. Hence, there exists $t \in \{1, \ldots , r \}$ such that $\{u,v \} \subseteq S_{t}$.  Consider the following cases.

\begin{description}
\item \textbf{Case 1}. $l(W)=3$

Since $D$ is an $r$-partite tournament, then either $(x_{1},v) \in A(D)$ or $(v,x_{1}) \in A(D)$. It follows from the fact that $W$ is a $uv$-geodesic that $(x_{1}, v) \notin A(D)$, which implies that $(v,x_{1}) \in A(D)$. A similar argument shows that $(x_{2},u) \in A(D)$. Hence, $(v,x_{1}, x_{2}, u)$ is a walk in $D$ with length at most $k-1$. By Lemma \ref{walks} (b) we have that $(v,u) \in A(C_{H}^{k-1}(D))$, contradicting the assumption on $(u,v)$, and the claim holds.

\item \textbf{Case 2.} $l(W) \geq 4$.

First, we will show that for every $i \in \{2, \ldots , n-2\}$, $x_{i} \in S_{t}$. Proceeding by contradiction, suppose that there exists $i \in \{ 2, \ldots , n-2 \}$ such that $x_{i} \notin S_{t}$. It follows from the fact that $D$ is an $r$-partite tournament that either $(u,x_{i}) \in A(D)$ or $(x_{i}, u) \in A(D)$. Since $W$ is a $uv$-geodesic, then $(u,x_{i} ) \notin A(D)$, which implies that $(x_{i}, u) \in A(D)$. A similar argument shows that $(v,x_{i}) \in A(D)$. Hence, $(v,x_{i},u)$ is a path in $D$ with length at most $k-1$ and, by Lemma \ref{walks} (b), we conclude that $(v,u) \in A(C_{H}^{k-1}(D))$, contradicting our assumption on $(u,v)$. Therefore, for every $i \in \{2, \ldots , n-2\}$, $x_{i} \in S_{t}$. Since $S_{t}$ is an independent set, we have that $n=4$.  

On the other hand, as $x_{1} \notin S_{t}$ and $D$ is an $r$-partite tournament, then we have that either $(v,x_{1}) \in A(D)$ or $(x_{1},v) \in A(D)$. Since $W$ is a $uv-$geodesic, then $(x_{1},v) \notin A(D)$, which implies that $(v,x_{1}) \in A(D)$. By applying an analogous argument on $x_{3}$ and $u$, we have that $(x_{3},u) \in A(D)$. Hence, $(v,x_{1}, x_{2}, x_{3}, u)$ is a $vu$-path with length at most $k-1$ and, by Lemma \ref{walks} (b), we conclude that $(v,u) \in A(C_{H}^{k-1}(D))$, contradicting the assumption on $(u,v)$
\end{description}

By the previous cases, we have that $d_{D}(u,v) \leq 2$. 
\end{proof}

\begin{theorem}
\label{panpart}
Let $D$ be an $H$-colored $r$-partite tournament. If $k \geq 5$, then $D$ has a $(k,H)$-kernel. 
\end{theorem}
\begin{proof}
First, we will show that every cycle in $C_{H}^{k-1}(D)$ has a symmetric arc, then, by applying Theorem \ref{duchet} and Lemma \ref{closure}, we will conclude that $D$ has a $(k,H)$-kernel. Proceeding by contradiction, suppose that there exists a cycle in  $C_{H}^{k-1}(D)$, say $C=(u_{0}, \ldots , u_{n})$, with no symmetric arcs.  For every $i \in \{0, \ldots , n-1\}$, consider a $u_{i}u_{i+1}$-geodesic in $D$, say $W_{i}$ (indices are taken modulo $n$). It follows from Lemma \ref{distrpartite} that for every $i \in \{ 0, \ldots , n-1\}$, $l(W_{i}) \leq 2$. Let $C'=\cup_{i=0}^{n-1}W_{i}$ and assume that $C' =(x_{0}, \ldots , x_{l})$. Notice that $C'$ is a closed walk and $x_{0}=x_{l} = u_{0}$.

\begin{description}
\item \textbf{Claim 1.} For every $i \in \{ l-3,l-2,l-1, l\}$, $(u_{1}, x_{i}) \notin A(D)$.

Proceeding by contradiction, suppose that $(u_{1}, x_{i}) \in A(D)$ for some $i \in \{ l-3, l-2,l-1,l\}$. Since $x_{l}=u_{0}$, we have that $i \neq l$. It follows that $(u_{1},x_{i}) \cup (x_{i}, C', x_{0})$ is a $u_{1}u_{0}$-walk in $D$ with length at most $k-1$, which implies that $(u_{1}, u_{0}) \in A(C_{H}^{k-1}(D))$  (Lemma \ref{walks} (b)), contradicting the choice of $C$.
\end{description}

On the other hand, since $u_{1} \in V(C')$, then $u_{1}$ has at least one out-neighbor in $C$. Consider $q=max \{ i \in \{ 1, \ldots , l \} : (u_{1}, x_{i}) \in A(D)\}$. Notice that $q \leq l-4$ (Claim 1).
In order to present our proof more compact, we will show the following claims.

	\begin{description}
	\item \textbf{Claim 2.} $(x_{q+2}, u_{1}) \in A(D)$.
	
	Proceeding by contradiction, suppose that $(x_{q+2}, u_{1}) \notin A(D)$. 
	By the choice of $q$, we have that $(u_{1}, x_{q+2}) \notin A(D)$, which implies that there exists $m \in \{1, \ldots , r \}$ such that $\{ u_{1}, x_{q+2} \} \subseteq S_{m}$. It follows from the fact that $D$ is an $r$-partite digraph, that either $(x_{q+3}, u_{1}) \in A(D)$ or $(u_{1}, x_{q+3}) \in A(D)$. By the choice of $q$, we conclude that $(x_{q+3},u_{1}) \in A(D)$. Notice that $(x_{q}, x_{q+1}, x_{q+2}, x_{q+3}, u_{1}, x_{q})$ is a $5$-cycle in $D$. 
Since every $u_{i}u_{i+1}$-geodesic has length at most $2$, then $(x_{q}, x_{q+1}, x_{q+2}, x_{q+3})$ has at least two consecutive vertices in $C$, say $u_{t}$ and $u_{t+1}$. It follows from Lemma \ref{walks} (c) that $(u_{t}, u_{t+1})$ is a symmetric arc in $C_{H}^{k-1}(D)$, which is no possible by the choice of $C$. Hence, the claim holds.	
\end{description}

	\begin{description}
	\item \textbf{Claim 3.} There exists $t \in \{0, \ldots , n-1\}$ such that $x_{q+1} = u_{t}$ and $x_{q+3} = u_{t+1}$.

First, we will show that  $(x_{q}, x_{q+1}, x_{q+2})$ has no two consecutive vertices of $C$...(*). Proceeding by contradiction, suppose that $(x_{q}, x_{q+1}, x_{q+2})$ has two consecutive vertices of $C$, say $u_{s}$ and $u_{s+1}$. Since $(u_{1}, x_{q}, x_{q+1}, x_{q+2}, u_{1})$ is a $4$-cycle in $D$, it follows from Lemma \ref{walks} (c) that $u_{r}$ and $u_{r+1}$ are joined by a symmetric arc in $C_{H}^{k-1}(D)$, which is no possible by the choice of $C$. Therefore, (*) holds.

On the other hand, from the choice of $C'$, we can conclude that there exists $t \in \{0, \ldots , n-1\}$ such that $A(W_{t}) \cap \{ (x_{q}, x_{q+1}),(x_{q+1}, x_{q+2}) \} \neq \emptyset$. Remark that $l(W_{t}) \leq 2$. It follows from (*) that $x_{q+1} = u_{t}$ and $x_{q+3} = u_{t+1}$, and the claim holds.
\end{description}
	
Now, consider the following cases:
	
	\begin{description}
	\item \textbf{Case 1.} $\{ u_{1}, x_{q+3} \} \subseteq S_{m}$ for some $m \in \{1, \ldots , r \}$.
	
	In this case, we have that $u_{1}$ and $x_{q+4}$ are adjacent in $D$, and by the choice of $q$, we can conclude that $(x_{q+4}, u_{1}) \in A(D)$. Hence, $(x_{q+3}, x_{q+4}, u_{1}, x_{q}, x_{q+1})$ is a $u_{t+1}u_{t}$-walk in $D$ with length at most $4$, which implies that $(u_{t+1}, u_{t}) \in A(C_{H}^{k-1}(D))$ (Lemma \ref{walks} (b)), contradicting the choice of $C$.
	
	\item \textbf{Case 2.} $u_{1} \in S_{m}$ and $x_{q+3} \in S_{m'}$ for some $\{m, m' \} \subseteq \{1, \ldots , r\}$, and $m \neq m'$.
	
	In this case, we have that $u_{1}$ and $x_{q+3}$ are adjacent in $D$, and by the choice of $q$, we conclude that $(x_{q+3}, u_{1}) \in A(D)$. Hence, $(x_{q+3}, u_{1}, x_{q}, x_{q+1})$ is a $u_{t+1}u_{t}$-walk in $D$ with length at most $3$, which implies that $(u_{t+1}, u_{t}) \in A(C_{H}^{k-1}(D))$ (Lemma \ref{walks} (b)), contradicting the choice of $C$.
	\end{description}

Therefore, we have that every cycle in $C_{H}^{k-1}(D)$ has a symmetric arc, which implies that $C_{H}^{k-1}(D)$ has a kernel (Theorem \ref{duchet}), concluding that $D$ has a $(k,H)$-kernel (Lemma \ref{closure}).
\end{proof}

\subsection{Local tournaments}

The following lemma will be useful in what follows.

\begin{lemma}\cite{localtour}
\label{localintour}
Every pair of vertices in each strong component of a local in-tournament (out-tournament) lie on a cycle.
\end{lemma}

As a consequence of the previous lemma, we have the following result in $H$-colored local in-tournaments.

\begin{theorem}
\label{localin}
Let $D$ be an $H$-colored local in-tournament. If every cycle in $D$ has $H$-length at most $k-2$ ($k \geq 2$), then $D$ has a $(k,H)$-kernel. 
\end{theorem}
\begin{proof}
First, we will show that every cycle in $C_{H}^{k-1}(D)$ has a symmetric arc, then, by applying Theorem \ref{duchet} and Lemma \ref{closure}, we will conclude that $D$ has a $(k,H)$-kernel. 

Let $C=(u_{0}, \ldots , u_{n})$ be a cycle in $C_{H}^{k-1}(D)$, and for every $i \in \{0, \ldots , n-1 \}$, consider a $u_{i}u_{i+1}$-path with $H$-length at most $k-1$, say $T_{i}$ (indices are taken modulo $n$). Clearly, $C'=\cup _{i=0}^{n-1} T_{i}$ is a closed walk in $D$ containing $u_{0}$ and $u_{1}$, which implies that $u_{0}$ and $u_{1}$ lie in the same strong component of $D$. It follows from Lemma \ref{localintour} that there exists a cycle in $D$, say $C_{0}$, such that $\{ u_{0}, u_{1} \} \subseteq V(C_{0})$. Suppose that $C_{0} =(u_{1} = x_{0}, \ldots, u_{0}=x_{t} , \ldots , x_{l} )$. If  $C_{0} ' =(u_{1}, C_{0}, u_{0})$, then $|O_{H}(C_{0}')| \leq |O_{H}(C_{0})|$ (Lemma \ref{subpaths}). By hypothesis, we conclude that $|O_{H}(C_{0}')| \leq k-2$, which implies that $l_{H}(C_{0}') \leq k-1$. Hence, $(u_{1}, u_{0}) \in A(C_{H}^{k-1}(D))$, concluding that $C$ has a symmetric arc in $C_{H}^{k-1}(D)$. 

Therefore, we have that every cycle in $C_{H}^{k-1}(D)$ has a symmetric arc, which implies that $C_{H}^{k-1}(D)$ has a kernel (Theorem \ref{duchet}), concluding that $D$ has a $(k,H)$-kernel (Lemma \ref{closure}).
\end{proof}

\begin{theorem}
\label{localout}
Let $D$ be an $H$-colored local out-tournament. If every cycle in $D$ has $H$-length at most $k-2$ ($k \geq 2$), then $D$ has a $(k,H)$-kernel. 
\end{theorem}
\begin{proof}
An analogous proof as in Theorem \ref{localin} will show Theorem \ref{localout}.
\end{proof}

\begin{cor}
Let $D$ be an $H$-colored local tournament. If every cycle in $D$ has $H$-length at most $k-2$ ($k \geq 2$), then $D$ has a $(k,H)$-kernel. 
\end{cor}

\section{A brief note on $(k, H)$-panchromatic digraphs}

Let $D$ be a digraph and $k \geq 2$, we say that $D$ is a \emph{$(k,H)$-panchromatic digraph} if for every digraph $H$ (possibly with loops), and every $H$-coloring of $D$, $D$ has a $(k,H)$-kernel. Previous work on panchromaticity in digraphs can be found in \cite{pan2}, \cite{pan1}, \cite{pan4} and \cite{pan3}. As a direct consequence of the results proved in this paper, we have that the following nearly tournaments are $(k,H)$-panchromatic for certain values of $k$. 

\begin{theorem}
If $D$ is a digraph, then $D$ is $(k, H)$-panchromatic provided that:
\begin{enumerate}[(i)]
\item (Theorem \ref{tour1}) $D$ is semicomplete and $k \geq 3$.

\item (Corollary \ref{pantrans})  $D$ is transitive and $k \geq 2$.

\item (Corollary \ref{panrtrans})  $D$ is $r$-transitive ($r \geq 2$) and $k \geq r$.

\item (Theorem \ref{panqker}) $D$ is quasi-transitive and $k \geq 4$

\item (Theorem \ref{pan3qtr}) $D$ is $3$-quasi-transitive and $k \geq 5$.

\item (Theorem \ref{panpart}) $D$ is an $r$-partite tournament ($r \geq 2$) and $k \geq 5$.
\end{enumerate}
\end{theorem}

\section*{Acknowledgments}
Hortensia Galeana-Sánchez is supported by CONACYT FORDECYT-PRONACES/39570/2020 and UNAM-DGAPA-PAPIIT IN102320. 
Miguel Tecpa-Galván is supported by CONACYT-604315.

\end{document}